\documentclass[11pt]{amsart}
\usepackage[T1]{fontenc}
\usepackage{amsmath,amssymb,amsthm}
\usepackage[margin=2.7cm]{geometry}
\usepackage[colorlinks=true,linkcolor=blue,citecolor=blue]{hyperref}
\usepackage{enumitem}
\setlist[enumerate,1]{label=\roman*., font=\normalfont}

\usepackage{xcolor}

\newtheorem{theorem}{Theorem}[section]
\newtheorem{proposition}[theorem]{Proposition}
\newtheorem{lemma}[theorem]{Lemma}
\newtheorem{corollary}[theorem]{Corollary}
\newtheorem*{theoremA}{Theorem A}
\newtheorem*{theoremB}{Theorem B}
\newtheorem*{theoremC}{Theorem C}
\newtheorem*{theoremD}{Theorem D}
\newtheorem*{theoremE}{Theorem E}
\newtheorem*{theoremF}{Theorem F}
\theoremstyle{remark}
\newtheorem{remark}[theorem]{Remark}

\DeclareMathOperator{\Gcd}{Gcd}
\DeclareMathOperator{\Ghd}{Ghd}
\DeclareMathOperator{\Gfd}{Gfd}
\DeclareMathOperator{\Gpd}{Gpd}
\DeclareMathOperator{\fd}{fd}
\DeclareMathOperator{\pd}{pd}
\DeclareMathOperator{\sfli}{sfli}
\DeclareMathOperator{\spli}{spli}
\DeclareMathOperator{\Hom}{Hom}
\DeclareMathOperator{\Ext}{Ext}
\DeclareMathOperator{\Tor}{Tor}
\DeclareMathOperator{\im}{im}
\DeclareMathOperator{\coker}{coker}
\DeclareMathOperator{\Ind}{Ind}
\DeclareMathOperator{\Coind}{Coind}
\DeclareMathOperator{\cd}{cd}
\DeclareMathOperator{\Gwgl}{Gwgl.dim}
\DeclareMathOperator{\wgl}{w.gl.dim}
\DeclareMathOperator{\gl}{gl.dim}

\newcommand{\PGFcd}{\widetilde{\Gcd}}
\newcommand{\Z}{\mathbb{Z}}
\newcommand{\Q}{\mathbb{Q}}
\newcommand{\D}{\mathrm{D}}
\newcommand{\LHF}{\mathbf{LH}\mathfrak{F}}

\begin{document}

\title{Gorenstein Homological Dimension of Group Extensions}

\author{Dimitra-Dionysia Stergiopoulou}
\address{Department of Mathematics, University of Thessaly, Lamia, Greece}
\email{dstergiop@math.uoa.gr, dstergiopoulou@uth.gr}

\begin{abstract}
We study the Gorenstein homological dimension $\Ghd_k G$ of groups $G$ which are of type $\mathrm{FP}_{\infty}$ over a commutative ring $k$. Our main result shows that, over an arbitrary ring $R$, every finitely presented Gorenstein flat $R$-module is projectively coresolved Gorenstein flat. Consequently, for a group $G$ of type $\mathrm{FP}_\infty$ over $k$ with $\sfli k<\infty$, the three natural dimensions for $G$, namely Gorenstein homological, Gorenstein cohomological and projectively coresolved Gorenstein flat, all coincide.
Building on this collapse, we establish a Gorenstein homological analogue of Fel$'$dman's theorem, a formula for iterated $m$-fold self-extensions of a group $N$ of type $\mathrm{FP}_\infty$ over $\Z$, and a field-detection theorem, showing that the Gorenstein homological dimension of a group of type $\mathrm{FP}_\infty$ over a principal ideal domain is realized after passing to a suitable field. 
\end{abstract}

\subjclass[2020]{Primary 20J05; Secondary 16E10, 18G25, 20J06}
\keywords{Gorenstein homological dimension, Gorenstein flat module, PGF module, group extension, characteristic module, complete resolution}

\maketitle

\section{Introduction}

Let $G$ be a group and $k$ be a commutative ring. The Gorenstein homological dimension $\Ghd_k G$ of $G$ over $k$ is defined as the Gorenstein flat dimension of the trivial module. It is the homological companion of the Gorenstein cohomological dimension $\Gcd_kG$ which has been studied by Emmanouil and Talelli \cite{ET18,ET22,ET25,Tal14}. It refines the classical homological dimension $\mathrm{hd}_kG$ in the way that $\Gcd$ refines $\mathrm{cd}$. Unlike $\mathrm{hd}_kG$, it can be finite for groups with torsion and unlike $\Gcd_kG$, it vanishes on all locally finite groups \cite[Corollary~5.3]{KS25}. Our recent work \cite{KS25} has made this invariant tractable: over a commutative ring of finite Gorenstein weak global dimension we established finiteness criteria in terms of weak characteristic modules and proved that $\Ghd_kG\le\Gcd_kG$. The present paper takes these criteria as its starting point.

Our first result is purely ring-theoretic. Recall that a module is PGF (projectively coresolved Gorenstein flat \cite{SS20}) if it is a syzygy of an acyclic complex of projectives that stays acyclic after tensoring with any injective right module. The class $\tt PGF(R)$ is contained in both $\tt GProj(R)$ and $\tt GFlat(R)$, and one has $\tt PGF(R)=\tt GProj(R)$ if and only if $\tt GProj(R)\subseteq\tt GFlat(R)$ \cite[p.~25]{SS20}. Whether the latter inclusion holds for every ring, i.e.\ whether every Gorenstein projective module is Gorenstein flat, is a well-known open problem and the class of PGF modules sits precisely between the two classes.

\begin{theoremA}[Theorem \ref{theo3.2}]
Let $R$ be a ring. Then, every finitely presented Gorenstein flat $R$-module is PGF, hence also Gorenstein projective.
\end{theoremA}
 
Holm's metatheorem \cite{Ho} claims that every result in classical homological algebra has a counterpart in Gorenstein homological algebra. Theorem A is a Gorenstein analogue of the classical fact that every finitely presented flat module is projective, with the PGF modules of \cite{SS20} playing the role of the projective modules. As in the classical case, the proof rests on the Govorov--Lazard theorem. Here, the finite presentation of the module is used to realize it as a syzygy of an F-totally acyclic complex of projective modules.
Theorem A applies, in particular, to the group algebra $kG$ of an arbitrary
group $G$. 

As a consequence of Theorem A, we deduce that on groups of type $\mathrm{FP}_\infty$ the Gorenstein homological and Gorenstein cohomological dimensions coincide. Let $k$ be a commutative ring and $G$ be a group. We denote by $\Gcd_k G$ (respectively, by $\Ghd_k G$) the Gorenstein projective dimension (respectively, the Gorenstein flat dimension) of the trivial $kG$-module $k$. We also denote by $\PGFcd_kG$ the projectively coresolved Gorenstein flat dimension of $G$ \cite{St24}.

\begin{theoremB}[Corollary \ref{cor3.5}]
Let $k$ be a commutative ring with $\sfli k<\infty$ and $G$ be a group of type $\mathrm{FP}_\infty$ over $k$. Then,
\[
\Ghd_k G=\Gcd_k G=\PGFcd_k G.
\]
\end{theoremB}

We note that the hypothesis on $G$ cannot be dropped. Indeed, an infinite locally finite group has $\Ghd_kG=0<\Gcd_kG$, for every commutative ring $k$. The vanishing $\Ghd_kG=0$ follows from $\Ghd_kG\le\Ghd_\Z G=0$ (see Proposition \ref{prop2.2}(iv) and \cite[Corollary~5.3]{KS25}), whereas $\Gcd_kG\ge1$ (see \cite[Corollary 2.3]{ET18}). Moreover, the hypothesis on $k$ is mild since $\sfli k\le\wgl k\le\gl k$. In \cite[Theorem~3.10]{KS25} we showed that over any commutative ring $k$ such that $\sfli k<\infty$ the inequality $\Ghd_k G\leq\Gcd_k G$ holds. Consequently, Corollary \ref{cor3.5} provides conditions under which the dimensions are equal.

Let $B(G,k)$ denote the $kG$-module of all functions $f:G\rightarrow k$ whose range is finite (see Section \ref{sec:fp} for its basic properties). Since it is not known in general when the dimensions $\fd_{kG}B(G,k)$ and $\pd_{kG}B(G,k)$ are equal, the following result sheds some light on the case where the group $G$ is an $\LHF$-group of type $\mathrm{FP}_\infty$.

\begin{theoremC}[Corollary \ref{cor3.6}]
	Let $k$ be a commutative ring such that $\spli k<\infty$ and let $G$ be an $\LHF$-group of type $\mathrm{FP}_\infty$ over $k$. Then,
	\[
	\Ghd_kG=\Gcd_kG=\PGFcd_kG=\fd_{kG}B(G,k)=\pd_{kG}B(G,k)<\infty.
	\]
\end{theoremC}

Using Corollary \ref{cor3.5} together with a Lyndon--Hochschild--Serre argument in homology, we obtain the following three theorems on group extensions, providing homological analogues of results known on the cohomological side \cite{ET22}.

\begin{theoremD}[Theorem \ref{theo5.9}]
Let $F$ be a field and let $1\to N\to G\to Q\to 1$ be a group extension, where $N$ is of type $\mathrm{FP}_\infty$ over $F$ and $\Ghd_F Q<\infty$. Then,
\[
\Ghd_F G=\Ghd_F N+\Ghd_F Q.
\]
\end{theoremD}

This is a Gorenstein homological analogue of Fel$'$dman's additivity theorem \cite{Fel71}. Over a general commutative ring with $\sfli k<\infty$ we show that $\Ghd_kG=\Ghd_k N+\Ghd_k Q$ (see Theorem \ref{theo5.7}) under the following hypothesis on the cohomology of the kernel with group ring coefficients, where $n=\Ghd_kN$: the $k$-modules $H^i(N,kN)$ are flat for $i<n$, and $H^n(N,kN)$ contains $k$ as a direct $k$-summand. This hypothesis cannot be omitted altogether, as Remark \ref{rem5.10} shows. We know that over any commutative ring $k$ of finite Gorenstein weak global dimension the inequality $\Ghd_k G\leq \Ghd_k N+\Ghd_k Q$ holds (see \cite[Proposition~4.4]{RY}). The reverse inequality is obtained by exhibiting an explicit injective $kG$-module, namely the Pontryagin dual $\D(kG^{(\alpha)})$, whose homology survives at the corner of the spectral sequence.

An analogue of Fel$'$dman's theorem for the projectively coresolved Gorenstein flat dimension of groups was obtained in \cite[Corollary~6.10]{St24}, where the additivity formula $\PGFcd_kG=\PGFcd_kN+\PGFcd_kQ$ is established for a kernel $N$ of type $\mathrm{FP}_\infty$ whose cohomology modules $H^i(N,kN)$ are $k$-free. Theorem \ref{theo5.7} is the corresponding statement for the Gorenstein homological dimension, and its proof uses a different approach based entirely on homological methods: the required nonvanishing is detected on the single explicit injective module $\D\big(kG^{(\alpha)}\big)$ through the Lyndon--Hochschild--Serre spectral sequence in homology, rather than through a complete resolution. Corollary~\ref{cor3.5} shows that $\operatorname{Ghd}_k G=\widetilde{\operatorname{Gcd}}_k G$ for every group $G$ of type $\mathrm{FP}_\infty$ over a ring $k$ with $\operatorname{sfli}k<\infty$. Thus, whenever the hypotheses of both results are satisfied, their conclusions concern the same invariant.

    Our next result shows that Gorenstein homological dimension
is additive under iterated self-extensions of groups of type
$\mathrm{FP}_\infty$ over $\Z$. As a special case, we obtain
a formula for the dimension of finite direct powers of such groups.

\begin{theoremE}[Theorem \ref{theo6.1}]
If $G$ is an iterated $m$-fold extension of a group $N$ of type $\mathrm{FP}_\infty$ over $\Z$ by itself, then $\Ghd_\Z G=m\,\Ghd_\Z N$. In particular, $\Ghd_\Z N^m=m\,\Ghd_\Z N$.
\end{theoremE}

Finally, we show that the Gorenstein homological dimension
of a group of type $\mathrm{FP}_\infty$ over a principal ideal
domain, when finite, can also be realized over a suitable field.

\begin{theoremF}[Theorem \ref{theo4.3}]
Let $k$ be a principal ideal domain and let $G$ be a group of type $\mathrm{FP}_\infty$ over $k$ with $\Ghd_k G<\infty$. Then, there exists a $k$-algebra $F$ which is a field, with $\Ghd_F G=\Ghd_k G$.
\end{theoremF}

Along the way we show that over a commutative ring $k$ of finite Gorenstein weak global dimension, a group of type $\mathrm{FP}_\infty$ such that $\Ghd_kG<\infty$ admits a characteristic module of type FP (Theorem \ref{theo4.1}), strengthening \cite[Theorem~3.1]{ET22} in two directions: by weakening the hypothesis on $k$ from $\wgl k<\infty$ to $\sfli k<\infty$ and the hypothesis on the finiteness of $\Gcd_kG$ to the finiteness of $\Ghd_kG$.

The paper is organized as follows. Section \ref{sec:prelim} collects the necessary background on Gorenstein modules, the invariants $\sfli$ and $\spli$, and weak characteristic modules. Section \ref{sec:fp} contains the proof of Theorem A. Using the Govorov--Lazard theorem inside a complete flat resolution, together with Bennis's extension criterion for Gorenstein flat modules \cite[Theorem~2.3]{Ben09} (see Proposition \ref{prop2.1}(i)), we embed every finitely presented Gorenstein flat module over an arbitrary ring into a finitely generated free module with finitely presented Gorenstein flat cokernel. Iterating this construction, we complete projective resolutions to F-totally acyclic complexes of projective modules, which proves Theorem A and, for groups of type $\mathrm{FP}_\infty$, produces complete resolutions consisting of finitely generated free modules and the identification of Theorems B and C. Section \ref{sec:char} uses these resolutions to construct characteristic modules of type FP and to prove the field-detection Theorem F. In Section \ref{sec:feldman} we establish the duality $H_q(N,\D V)\cong\D H^q(N,V)$ for kernels $N$ of type $\mathrm{FP}_\infty$ and prove Theorem D, by computing the corner term of the Lyndon--Hochschild--Serre spectral sequence in homology with coefficients in the explicit injective module $\D(kG^{(\alpha)})$. Dranishnikov's examples \cite{Dra97} show that the flatness hypothesis of Theorem \ref{theo5.7} cannot be removed. Finally, Section \ref{sec:appl} applies the preceding results to iterated self-extensions, proving Theorem E.

\subsection*{Conventions} All rings are associative with unit and all modules are left modules, unless stated otherwise. For a group $G$ and a commutative ring $k$ we write $kG$ for the group algebra; the anti-isomorphism $g\mapsto g^{-1}$ yields $kG\cong (kG)^{\mathrm{op}}$, so there is no essential distinction between left and right $kG$-modules. We write $\D(-)=\Hom_\Z(-,\Q/\Z)$ for the Pontryagin dual.

\section{Preliminaries}\label{sec:prelim}
This section recalls basic background in Gorenstein homological algebra which  will be needed throughout this paper.

\subsection{Gorenstein modules}

Let $\mathbf{P}$ be an exact complex of projective $R$-modules. The complex $\mathbf{P}$ is called \emph{totally acyclic} if $\operatorname{Hom}_R(\mathbf{P},Q)$
is exact for every projective $R$-module $Q$. An $R$-module is said to be \emph{Gorenstein projective} if it occurs as a syzygy in a totally acyclic complex of projective modules. We denote the class of all Gorenstein projective $R$-modules by $\tt GProj(R)$.
For an $R$-module $M$, its \emph{Gorenstein projective dimension}, denoted by
$\operatorname{Gpd}_R M$, is the least integer $n\geq 0$ for which there exists an exact sequence $0\rightarrow G_n\rightarrow G_{n-1}\rightarrow\cdots
\rightarrow G_0\rightarrow M\rightarrow 0$,
where each $G_i$ is Gorenstein projective. When no finite resolution of this form exists, one sets $
\operatorname{Gpd}_R M=\infty$.
We refer to \cite{Holm04} for further details.

Similarly, an exact complex $\mathbf{F}$ of flat $R$-modules is called \emph{totally acyclic in the flat sense} or \emph{F-totally acyclic} if the complex $I\otimes_R\mathbf{F}$
is exact for every injective right $R$-module $I$. A module is called Gorenstein flat if it is a syzygy in such a complex. The class of Gorenstein flat $R$-modules will be denoted by $\tt GFlat(R).$ The Gorenstein flat dimension of an $R$-module $M$, denoted by $\operatorname{Gfd}_R M$,
is the minimum length of a resolution of $M$ whose terms are Gorenstein flat modules. If $M$ admits no finite Gorenstein flat resolution, then $\operatorname{Gfd}_R M=\infty$. 

By \cite[Corollary~4.12]{SS20}, every ring is GF-closed (the class $\tt GFlat(R)$ is closed under extensions). Consequently, the following fundamental results of Bennis hold, and they will both be used throughout the paper. We state them separately, since they are invoked for different purposes: part (i) is the extension criterion, used in Section \ref{sec:fp} to produce Gorenstein flat cokernels, whereas part (ii) is the $\Tor$-characterization of the Gorenstein flat dimension, used in Section \ref{sec:feldman} to detect nonvanishing homology. We stress that part (ii) requires the finiteness of the Gorenstein flat dimension.

\begin{proposition}[{\cite[Theorems~2.3 and~2.8]{Ben09}, \cite[Corollary~4.12]{SS20}}]\label{prop2.1} Let $R$ be a ring.
\begin{enumerate}[label=(\roman*)]
\item For every short exact sequence of left $R$-modules $0\rightarrow G_1\rightarrow G_0 \rightarrow M \rightarrow 0,$ where $G_0$ and $G_1$ are Gorenstein flat, if $\operatorname{Tor}_1^R(I,M)=0$ for every injective right $R$-module $I$, then $M$ is Gorenstein flat.
\item Let $M$ be a nonzero left $R$-module with $\Gfd_RM<\infty$. Then,
\[
\Gfd_RM=\sup\bigl\{i:\Tor_i^R(I,M)\neq0\ \text{ for some injective right }R\text{-module }I\bigr\}.
\]
In particular, writing $n=\Gfd_RM$, we have $\Tor_i^R(I,M)=0$ for every $i>n$ and every injective right $R$-module $I$, while $\Tor_n^R(I,M)\neq0$ for at least one injective right $R$-module $I$.
\end{enumerate}
\end{proposition}

We shall also use the class of \emph{projectively coresolved Gorenstein flat modules}, usually abbreviated to PGF-modules. This class of modules was introduced by \v{S}aroch and \v{S}\v{t}ov\'{\i}\v{c}ek in \cite{SS20}. An $R$-module is PGF if it is a cycle in an exact complex $\mathbf{P}$ of projective modules satisfying the additional condition that $I\otimes_R\mathbf{P}$
is exact for every injective right $R$-module $I$.
The class of all PGF $R$-modules is denoted by $\tt PGF(R).$
Since every projective module is flat, one immediately obtains that $\tt PGF(R)\subseteq \tt GFlat(R).$
On the other hand, the inclusion $\tt PGF(R)\subseteq \tt GProj(R)$
is established in \cite[Theorem~4.4]{SS20}. The class $\tt PGF(R)$ enjoys several useful closure properties: it is closed under extensions, arbitrary direct sums, direct summands, and kernels of epimorphisms. Finally, the PGF-dimension of an $R$-module $M$, denoted by $\operatorname{PGF\text{-}dim}_R M$, is defined as the least length of a resolution of $M$ by PGF-modules. As usual, if no such finite resolution exists, we put $\operatorname{PGF\text{-}dim}_R M=\infty.$ We shall also need our result \cite[Lemma~2.1]{KS25}, which states that \emph{every PGF module of finite flat dimension is projective}.

\subsection{Group rings, Gedrich--Gruenberg invariants and Gorenstein global dimensions} Let $k$ be a commutative ring and let $G$ be a group. We write $kG$ for the corresponding group ring. Our main reference for the basic notions and results concerning group cohomology is \cite{Brown}. As noted in the Conventions, $kG\cong(kG)^{\operatorname{op}}$; in particular, every right $kG$-module $M$ may be regarded as a left $kG$-module by defining $g\cdot m=mg^{-1}$ for all $g\in G$ and $m\in M$.

For an arbitrary ring $R$, the homological invariant $\operatorname{spli}(R)$ was introduced by Gedrich and Gruenberg in
\cite{GG}. It is defined by
$\operatorname{spli}(R)=\sup\bigl\{\operatorname{pd}_R I\mid I \text{ is an injective } R\text{-module}\bigr\}.$ Thus, $\operatorname{spli}(R)$ measures
the largest possible projective dimension of an injective module.
The invariant $\sfli R$ is defined analogously as the supremum of the flat dimensions of injective (left) $R$-modules. The Gorenstein weak global dimension $\Gwgl R$ of a ring $R$ is defined as $\Gwgl R=\sup\bigl\{\operatorname{Gfd}_R M\mid M \text{ is an } R\text{-module}\bigr\}$. By \cite[Theorem~2.4]{CET21}, the Gorenstein weak global dimension is symmetric, i.e.\ $\Gwgl R=\Gwgl R^{\mathrm{op}}$. Moreover, by \cite[Theorem~5.3]{Emm12} we have $\Gwgl R<\infty\iff\sfli R<\infty\ \text{and}\ \sfli R^{\mathrm{op}}<\infty$, in which case $\Gwgl R=\sfli R=\sfli R^{\mathrm{op}}$. For $R\cong R^{\mathrm{op}}$ (e.g.\ $R=kG$) the finiteness of $\Gwgl R$ is equivalent to the finiteness of $\sfli R$. We note that $\sfli k\le\wgl k\le\gl k$, so our hypotheses are weaker than those of \cite{ET22}. Finally, $\sfli(kH)\le\sfli(kG)$ for every subgroup $H\le G$, with equality when $[G:H]<\infty$ \cite[Lemma~2.10]{KS25}.

\subsection{Kropholler's hierarchy} (see \cite{Kro93}) Let $\mathfrak{F}$ be the class of finite groups. We define $\mathbf{H}_0\mathfrak{F}:=\mathfrak{F}$ and for every ordinal number $\alpha>0$ we say that a group $G$ belongs to the class $\mathbf{H}_{\alpha}\mathfrak{F}$ if and only if there exists a finite dimensional contractible CW-complex on which $G$ acts such that every isotropy subgroup of the action belongs to $\mathbf{H}_{\beta}\mathfrak{F}$ for some ordinal $\beta < \alpha$. We say that a group $G$ belongs to $\mathbf{H}\mathfrak{F}$ if and 
only if there is an ordinal $\alpha$ such that $G$ belongs to $\mathbf{H}_{\alpha}\mathfrak{F}$. Moreover, we define a group $G$ to be in  $\LHF$ if and only if all finitely generated subgroups of $G$ are in $\mathbf{H}\mathfrak{F}$.

\subsection{(Weak) characteristic modules}
Let $k$ be a commutative ring and $G$ be a group. A \emph{weak characteristic module} for $G$ over $k$ is a $k$-flat $kG$-module $A$ with $\fd_{kG}A<\infty$, admitting a $k$-pure $kG$-linear monomorphism $\iota:k\to A$ (with $k$ the trivial $kG$-module) {\cite[Definition~3.6]{KS25}}. Weak characteristic modules generalize the notion of characteristic modules, i.e. a $k$-projective $kG$-module $A$ with $\textrm{pd}_{kG}A <\infty$, which admits a $k$-split $kG$-linear monomorphism $\iota: k \rightarrow A$ (see \cite[Definition~1.1]{ET22}, \cite[Remark~3.9]{KS25}).


\subsection{Basic finiteness properties}\label{subsec:properties}

We collect here the basic finiteness properties of the Gorenstein homological dimension of groups:

\begin{proposition}\label{prop2.2}
Let $k$ be a commutative ring and $G$ be a group.
\begin{enumerate}[label=(\roman*)]
\item If $\sfli k<\infty$ and $\Ghd_kG<\infty$, then every $kG$-module has finite Gorenstein flat dimension (\cite[Theorem~3.14]{KS25}).
\item If $\sfli k<\infty$ and $H\le G$, then $\Ghd_kH\le\Ghd_kG$ (\cite[Proposition~4.1]{RY}).
\item If $\sfli k<\infty$ and $N\trianglelefteq G$ with $Q=G/N$, then $\Ghd_kG\;\le\;\Ghd_kN+\Ghd_kQ$ (\cite[Proposition~4.4]{RY}).
\item $\Ghd_kG\le\Ghd_\Z G$ for every commutative ring $k$ (\cite[Proposition~3.8]{RY}).
\item If $G=\varinjlim_\lambda G_\lambda$ is the directed union of a family of subgroups $\{G_\lambda\}$ of $G$, then $\Ghd_kG\le\sup_\lambda\Ghd_kG_\lambda$ (\cite[Proposition~5.14]{KS25}).
\end{enumerate}
\end{proposition}

In addition, we have $\Ghd_{\mathbb{Z}}G=0\iff G$ is locally finite (\cite[Corollary~5.3]{KS25}) (whereas $\Gcd_kG=0\iff G$ is finite \cite{ET18}), and the analogue of Serre's theorem: if $\sfli k<\infty$ and $[G:H]<\infty$, then $\Ghd_kG=\Ghd_kH$ (\cite[Theorem~5.12]{KS25}).

\section{Groups of type \texorpdfstring{$\mathrm{FP}_\infty$}{FP-infinity} with finite \texorpdfstring{$\Ghd$}{Ghd}}\label{sec:fp}

In \cite[\S2]{ET22} it is proved that every group of type $\mathrm{FP}_\infty$ with finite $\Gcd$ admits a complete projective resolution consisting of finitely generated free modules. Here, under the hypothesis $\Ghd_kG<\infty$ (which is weaker than the hypothesis $\Gcd_kG<\infty$ whenever $\sfli k<\infty$, by \cite[Theorem~3.10]{KS25}), we construct a complete resolution of finitely generated free modules which is F-totally acyclic. Our construction then yields Theorems A and B.


\begin{proposition}\label{prop3.1}
Let $R$ be a ring and let $M$ be a finitely presented Gorenstein flat $R$-module. Then:
\begin{enumerate}[label=(\roman*)]
\item There exists a short exact sequence of $R$-modules $0\to M\to T\to C\to 0$, where $T$ is finitely generated free, $C$ is finitely presented Gorenstein flat, and the sequence remains exact after applying the functor $I\otimes_R-$ for every injective right $R$-module $I$.
\item There exists an exact sequence of $R$-modules
\[
0\to M\to T_{-1}\to T_{-2}\to\cdots,
\]
where the $T_i$ are finitely generated free and the images $C_i=\im(T_i\to T_{i-1})$ are finitely presented Gorenstein flat modules. 
\item Every projective resolution $\mathbf P$ of $M$ 
\[ \cdots \rightarrow P_m \rightarrow P_{m-1}\rightarrow \cdots \rightarrow P_0 \rightarrow M \rightarrow 0\]
can be completed to an F-totally acyclic complex of projective modules
\[
\mathbf T:\ \cdots\rightarrow P_m \rightarrow P_{m-1}\rightarrow \cdots \rightarrow P_0\to T_{-1}\to T_{-2}\to\cdots\] with $M=\im(P_0\to T_{-1})$, where $T_i$ is a finitely generated free $R$-module for every $i\le-1$.
\end{enumerate}
\end{proposition}

\begin{proof}
(i) Since $M$ is Gorenstein flat, there exists a complete flat resolution of $R$-modules \[\mathbf F:\cdots \rightarrow F_1 \rightarrow F_0 \rightarrow F_{-1}\rightarrow F_{-2}\rightarrow \cdots\] such that $M=\im(F_1\to F_0)$; set $M'=\im(F_0\to F_{-1})$, so that we have a short exact sequence of $R$-modules $0\to M\xrightarrow{u}F_0\to M'\to 0$, where $F_0$ is flat and $M'$ is Gorenstein flat. Let $I$ be an injective right $R$-module. Since $M'$ is Gorenstein flat, we have $\Tor_1^R(I,M')=0$ and hence the long exact $\Tor$ sequence \[\cdots \rightarrow \Tor_1^R(I,M')\rightarrow I\otimes_R M \xrightarrow{1_I\otimes u} I \otimes_R F_0\rightarrow I\otimes_R M' \rightarrow 0\] gives that the map $1_I\otimes u$ is a monomorphism. By the Govorov--Lazard theorem, $F_0=\varinjlim_{j}L_j$ is a directed colimit of finitely generated free $R$-modules. Since $M$ is finitely presented, the canonical map $\varinjlim_j\Hom_R(M,L_j)\to\Hom_R(M,F_0)$ is an isomorphism. Hence, there exist an index $j$ and a map $g:M\to L_j$ such that $u=u_j\circ g$, where $u_j:L_j\to F_0$ is the structural map. Since $u$ is a monomorphism, $g$ is also a monomorphism. Set $T=L_j$ and $C=\coker g$. It follows that $C$ is a finitely presented $R$-module.
Since the composition $I\otimes_R M\xrightarrow{1_I\otimes g}I\otimes_R T\xrightarrow{1_I\otimes u_j}I\otimes_R F_0$ equals the monomorphism $1_I\otimes u$, it follows that $1_I\otimes g$ is also a monomorphism. We consider now the long exact $\Tor$ sequence of the short exact sequence $0\rightarrow M\xrightarrow{g} T\rightarrow C\rightarrow 0$: 
\begin{equation}\label{tor}
\cdots \rightarrow \Tor_1^R(I,T)\rightarrow \Tor_1^R(I,C)\rightarrow I\otimes_R M \xrightarrow{1_I\otimes g} I \otimes_R T\rightarrow I\otimes_R C \rightarrow 0.
\end{equation}
Since the $R$-module $T$ is free, the vanishing of $\Tor_{i\geq 1}^R(I,T)$ yields $\Tor_1^R(I,C)=\ker(1_I\otimes g)=0$. Consequently, the sequence $0\rightarrow I\otimes_RM\rightarrow I\otimes_RT\rightarrow I\otimes_RC\rightarrow 0$ is exact
for every injective right $R$-module $I$.  Moreover, the $R$-modules $T$ and $M$ are Gorenstein flat and hence the short exact sequence $0\rightarrow M\xrightarrow{g} T\rightarrow C\rightarrow 0$ and Proposition \ref{prop2.1}(i) yield that the $R$-module $C$ is Gorenstein flat, as needed.

(ii) This follows inductively from (i), splicing the short exact sequences obtained by a repeated application of (i). We note that the module $C$ of each step is finitely presented and Gorenstein flat, so the step can be repeated.

(iii) We splice the projective resolution $\mathbf P$ of $M$ with the exact sequence of (ii). The resulting complex is an acyclic complex of projective modules \[\mathbf T:\ \cdots\rightarrow P_m \rightarrow P_{m-1}\rightarrow \cdots \rightarrow P_0\to T_{-1}\to T_{-2}\to\cdots\] with $M=\im(P_0\to T_{-1})$, such that $T_i$ is a finitely generated free $R$-module for every $i\le-1$. Its syzygies in nonnegative degrees are the syzygies $K_j$ of $\mathbf P$, where $K_0=M$. Since $M$ is Gorenstein flat, invoking \cite[Corollary~4.12]{SS20} and \cite[Theorem 2.8(4)]{Ben09}, we deduce that $K_j$ is Gorenstein flat for every $j\geq 0$ and hence $\Tor_1^R(I,K_j)=0$. In negative degrees the syzygies of $\mathbf T$ are the Gorenstein flat modules $C_i$ of (ii), and hence $\Tor_1^R(I,C_i)=0$. We denote by $Z_i$, $i\in\mathbb{Z}$, the syzygies of $\mathbf T$. Then, every short exact sequence $0\to Z_i\to T_i\to Z_{i-1}\to0$  remains exact after the application of the functor $I\otimes_R-$, since $\Tor_1^R(I,Z_{i-1})=0$ for every $i\in \mathbb{Z}$. It follows that $\mathbf T$ is F-totally acyclic.
\end{proof}

\begin{theorem}[Theorem A]\label{theo3.2}
Let $R$ be an arbitrary ring. Then, every finitely presented Gorenstein flat $R$-module $M$ is PGF, hence also Gorenstein projective. If, moreover, $M$ is of type $\mathrm{FP}_\infty$, the complex $\mathbf T$ of Proposition \ref{prop3.1}(iii) can be chosen to consist of finitely generated free modules in \emph{all} degrees.
\end{theorem}

\begin{proof}
By Proposition \ref{prop3.1}(iii), $M$ is a syzygy of an acyclic complex of projective modules which remains acyclic after applying $I\otimes_R-$ for every injective $I$. It follows that $M$ is PGF and hence also Gorenstein projective.  If, moreover, $M$ is of type $\mathrm{FP}_\infty$, we can choose the projective resolution $\mathbf P$ to consist of finitely generated free modules. Thus, the complex $\mathbf T$ of Proposition \ref{prop3.1}(iii) can be chosen to consist of finitely generated free modules in \emph{all} degrees.
\end{proof}

\begin{corollary}\label{cor3.3}Let $R$ be a ring and let $M$ be an $R$-module of type $\mathrm{FP}_\infty$ such that $\Gfd_R M=s<\infty$. Then, for every resolution $\mathbf P\to M$ by finitely generated free $R$-modules, there exists an acyclic complex $\mathbf T$ of finitely generated free $R$-modules in all degrees, which:
\begin{enumerate}[label=(\alph*)]
\item coincides with $\mathbf P$ in degrees $\ge s$,
\item is F-totally acyclic (so that all its syzygies are PGF), and
\item is such that $\Hom_{R}(\mathbf T,Q)$ is acyclic for every projective $R$-module $Q$, i.e.\ $\mathbf T$ is a complete projective resolution of coincidence index $s$ in the strong sense of \cite[\S1.I]{ET22}.
\end{enumerate}
In particular, if $k$ is a commutative ring and $G$ is a group of type $\mathrm{FP}_\infty$ over $k$ such that $\Ghd_kG=n<\infty$, the trivial module $k$ admits a complete projective resolution consisting of finitely generated free modules in every degree, of coincidence index $n$, which is moreover F-totally acyclic.
\end{corollary}

\begin{proof} Let $K_s$ be the $s$-th syzygy of $\mathbf P$. Then, $K_s$  is of type $\mathrm{FP}_\infty$, hence finitely presented, and Gorenstein flat (see \cite[Theorem 2.8]{Ben09}). Applying Theorem \ref{theo3.2} to $K_s$ with left resolution the truncation $\mathbf P_{\ge s}$, we obtain the desired acyclic complex $\mathbf T$ with the properties (a) and (b). For (c), the syzygies of $\mathbf T$ are PGF and hence Gorenstein projective \cite[Theorem~4.4]{SS20}. Consequently, for every short exact sequence $0\to Z\to T\to Z'\to0$ of syzygies of $\mathbf T$ we have $\Ext^1_{R}(Z',Q)=0$ for every projective $R$-module $Q$. We conclude that $\Hom_{R}(\mathbf T,Q)$ is acyclic for every projective $R$-module $Q$, as needed.
\end{proof}

\begin{corollary}\label{cor3.4}Let $R$ be a ring and let $M$ be an $R$-module of type $\mathrm{FP}_\infty$. Then, $$\Gpd_RM\leq \operatorname{PGF-dim}_RM=\Gfd_RM.$$
\end{corollary}

\begin{proof}The inequalities $\Gpd_RM\leq \operatorname{PGF-dim}_RM$ and $\Gfd_RM\leq \operatorname{PGF-dim}_RM$ follow from the inclusions ${\tt PGF}(R)\subseteq {\tt GFlat}(R)$ and ${\tt PGF}(R)\subseteq {\tt GProj}(R)$ (see \cite[Theorem 4.4]{SS20}). If $\Gfd_RM=\infty$, there is nothing to prove. Let $\Gfd_RM=n<\infty$ and $\mathbf P \rightarrow M$ be a resolution by finitely generated free modules. Then, the $n$-th syzygy $K_n$ of $\mathbf P$ is Gorenstein flat (see \cite[Theorem 2.8]{Ben09}) and it is also finitely presented since $M$ is of type $\mathrm{FP}_\infty$. Applying Theorem \ref{theo3.2}, we deduce that $K_n\in{\tt PGF}(R)$. Consequently, $\operatorname{PGF-dim}_RM\leq n$ and hence $\operatorname{PGF-dim}_RM=\Gfd_RM=n$, as needed.
\end{proof}

\begin{corollary}[Theorem B]\label{cor3.5}
Let $k$ be a commutative ring and $G$ be a group of type $\mathrm{FP}_\infty$ over $k$. Then, $\Gcd_kG\leq\Ghd_kG=\PGFcd_kG.$ If, moreover, $\sfli k<\infty$, then $$\Gcd_kG=\Ghd_kG=\PGFcd_kG.$$
\end{corollary}

\begin{proof} The first assertion follows from Corollary \ref{cor3.4} applied to the trivial $kG$-module $k$. For the second assertion, we combine the known inequality $\Ghd_kG\leq\Gcd_kG$ (see \cite[Theorem 3.10]{KS25}) with the first inequality $\Gcd_kG\leq\Ghd_kG$ and deduce that $\Gcd_kG=\Ghd_kG=\PGFcd_kG,$ as needed.\end{proof}

We recall that $B(G,k)$ denotes the $kG$-module of all functions $f:G\rightarrow k$ whose range is finite. This module is free as a $k$-module and its restriction to $kH$ is free for every finite subgroup $H\leq G$. For every element $\lambda \in k$, the constant function $\iota(\lambda)\in B(G,k)$ with value $\lambda$ is invariant under the action of $G$. The map $\iota: k \rightarrow B(G,k)$ defined in this way is $kG$-linear and $k$-split. Furthermore, the cokernel $\overline{B}(G,k)$ of $\iota$ is $k$-free (see \cite[Lemma~3.3]{CK98} and \cite[Lemma~3.4]{BC}). The module $B(G,k)$ provides a natural candidate for a (weak) characteristic module for an arbitrary group $G$ over a commutative ring $k$.

\begin{corollary}[Theorem C]\label{cor3.6}
	Let $k$ be a commutative ring such that $\spli k<\infty$ and let $G$ be an $\LHF$-group of type $\mathrm{FP}_\infty$ over $k$. Then,
	\[
	\Ghd_kG=\Gcd_kG=\PGFcd_kG=\fd_{kG}B(G,k)=\pd_{kG}B(G,k)<\infty.
	\]
\end{corollary}

\begin{proof} The finiteness of $\spli k$, together with \cite[Theorem~7.6]{St24}, yields $\Gcd_kG=\PGFcd_kG=\pd_{kG}B(G,k)<\infty$. Moreover, $\sfli k\le\spli k<\infty$ and hence $\Ghd_kG=\fd_{kG}B(G,k)<\infty$ by \cite[Corollary~6.11]{St24b}. Using Corollary \ref{cor3.5} we deduce that $\Ghd_kG=\Gcd_kG=\PGFcd_kG=\fd_{kG}B(G,k)=\pd_{kG}B(G,k)<\infty$, as needed.
\end{proof}

\begin{corollary}\label{cor3.7}
	Let $G$ be an $\LHF$-group of type $\mathrm{FP}_\infty$ over $\mathbb{Z}$. Then,
	\[
	\Ghd_\mathbb{Z}G=\Gcd_\mathbb{Z}G=\PGFcd_\mathbb{Z}G=\fd_{\mathbb{Z}G}B(G,\mathbb{Z})=\pd_{\mathbb{Z}G}B(G,\mathbb{Z})<\infty.
	\]
\end{corollary}

\section{Characteristic modules for type \texorpdfstring{$\mathrm{FP}_\infty$}{FP-infinity} groups with finite \texorpdfstring{$\Ghd$}{Ghd}}\label{sec:char}

Characteristic modules, introduced in \cite{ET22}, are the main computational device for the Gorenstein cohomological dimension: for every characteristic module $A$ for $G$ over $k$ one has $\Gcd_kG\le\pd_{kG}A$ (see \cite[Proposition~1.2, Corollary~1.3]{ET22}). In this section we use the complete resolutions of Section \ref{sec:fp} to construct, for every group of type $\mathrm{FP}_\infty$ with $\Ghd_kG<\infty$ over a commutative ring $k$ with $\sfli k<\infty$, a characteristic module of type FP whose projective dimension equals $\Ghd_kG=\Gcd_kG$ (Theorem \ref{theo4.1}). This result strengthens \cite[Theorem~3.1]{ET22}, as we explain in the remark following the proof. The finiteness properties of this module make it well suited to base change. As a consequence, we obtain Theorem F: the Gorenstein homological dimension of a group of type $\mathrm{FP}_\infty$ over a principal ideal domain is detected after passage to a suitable field.

\begin{theorem}\label{theo4.1}
Let $k$ be a commutative ring such that $\sfli k<\infty$ and let $G$ be a group of type $\mathrm{FP}_\infty$ over $k$ with $\Ghd_kG=n<\infty$. Then, there exists a characteristic module $A$ for $G$ over $k$ of type FP such that
\[
\pd_{kG}A=\Gcd_kG=\Ghd_kG=n.
\]
\end{theorem}

\begin{proof} Let $\mathbf P\to k$ be a resolution by finitely generated free $kG$-modules. By Corollary \ref{cor3.3} (case $M=k$, $s=n$) there exists a complete projective resolution $\mathbf T$ consisting of finitely generated free modules in every degree, of coincidence index $n$, with $T_i=P_i$ for $i\ge n$, such that $\mathbf T$ is F-totally acyclic and the complex $\Hom_{kG}(\mathbf T,Q)$ is acyclic for every projective $kG$-module $Q$. We now invoke the construction used in the proof of \cite[Theorem~3.1]{ET22}. We note that this construction is applied here to the complex $\mathbf T$ produced above, and that it uses only the following three properties of $\mathbf T$: (a) $\mathbf T$ consists of finitely generated free $kG$-modules in every degree, (b) $\mathbf T$ is a complete projective resolution of the trivial module $k$ of coincidence index $n$, i.e.\ $T_i=P_i$ for $i\ge n$, and (c) the complex $\Hom_{kG}(\mathbf T,Q)$ is acyclic for every projective $kG$-module $Q$. In particular, the hypotheses $\wgl k<\infty$ and $\Gcd_kG<\infty$ under which \cite[Theorem~3.1]{ET22} is stated are used there \emph{only} to produce a complex with the properties (a)--(c), which Corollary \ref{cor3.3} supplies here under the weaker hypotheses $\sfli k<\infty$ and $\Ghd_kG<\infty$; the construction is independent of how $\mathbf T$ was obtained. It produces a chain map $\tau:\mathbf T\to\mathbf P$ which is the identity in degrees $\ge n$, then modifies this map so that its components are surjective and forms a pushout, yielding a short exact sequence of $kG$-modules
	\begin{equation}\label{eq2}
		0\to k\xrightarrow{\ \iota\ }A\to M\to0,\qquad M=\coker(T_0\to T_{-1}),
	\end{equation}
together with an exact sequence of $kG$-modules \[
0\to Q_{n-1}\to\cdots\to Q_0\to T_{-1}\to A\to0,
\] where $T_{-1}$ and the $Q_i$'s are finitely generated and projective. Consequently, the $kG$-module $A$ is of type FP and $\pd_{kG}A\leq n <\infty$. The $kG$-module $M$ is PGF, since it is a syzygy of the F-totally acyclic complex $\mathbf T$. Invoking \cite[Lemma~2.1]{St24}, we infer that $M$ is a PGF $k$-module. Since $\fd_kA\le\fd_{kG}A\le\pd_{kG}A\le n<\infty$, the short exact sequence \eqref{eq2} yields $\fd_kM<\infty$. Consequently, $M$ is a $k$-projective module by \cite[Lemma~2.1]{KS25}. Then, $\Ext^1_k(M,k)=0$, so that the short exact sequence \eqref{eq2} is $k$-split and $A$ is a $k$-projective module. Hence, $A$ is a characteristic module of type FP. Since $A$ is a characteristic module we have $\Gcd_kG\le\pd_{kG}A$ (see \cite[Proposition~1.2, Corollary~1.3]{ET22}). Using the equality $\Gcd_kG=\Ghd_kG=n$ by Corollary \ref{cor3.5}, we conclude that $\pd_{kG}A=\Gcd_kG=\Ghd_kG=n$, as needed.
\end{proof}

\begin{remark}\label{rem4.2}
Theorem \ref{theo4.1} \emph{strengthens} \cite[Theorem~3.1]{ET22} in two ways: the hypothesis on the ring is relaxed from $\wgl k<\infty$ to $\sfli k<\infty$, and the hypothesis on the group from $\Gcd_kG<\infty$ to $\Ghd_kG<\infty$ (which coincide for $\mathrm{FP}_\infty$ groups).
\end{remark}

\begin{theorem}[Theorem F]\label{theo4.3}
Let $k$ be a principal ideal domain and let $G$ be a group of type $\mathrm{FP}_\infty$ over $k$ such that $\Ghd_kG<\infty$. Then, there exists a $k$-algebra $F$ which is a field, such that $\Ghd_FG=\Ghd_kG.$
\end{theorem}

\begin{proof}
Since $k$ is a principal ideal domain, we have $\sfli k\le\gl k\le1<\infty$. Then, Corollary \ref{cor3.5} yields $\Gcd_kG=\Ghd_kG<\infty$. By \cite[Theorem~3.6]{ET22} there exists a $k$-algebra $F$ which is a field such that $\Gcd_FG=\Gcd_kG$. The group $G$ is of type $\mathrm{FP}_\infty$ over $F$ as well (see \cite[\S1.II]{ET22}) and $\sfli F=0$. Hence Corollary \ref{cor3.5} yields again $\Ghd_FG=\Gcd_FG$. It follows that $\Ghd_FG=\Gcd_FG=\Gcd_kG=\Ghd_kG.$
\end{proof}

\begin{remark}\label{rem4.4}
An alternative proof of Theorem \ref{theo4.3} may be given as follows: Theorem \ref{theo4.1} yields a characteristic module $A$ of type FP such that $\pd_{kG}A=\Ghd_kG$. Using \cite[Proposition~3.5]{ET22} we can find a $k$-algebra $F$ which is a field such that $\pd_{FG}(A\otimes_kF)=\pd_{kG}A$. Since the module $A\otimes_kF$ is a characteristic module for $G$ over $F$, we have $\Gcd_FG=\pd_{kG}A$ as in the proof of \cite[Theorem~3.6]{ET22} and the result follows from Corollary \ref{cor3.5} as before. 
\end{remark}

\begin{corollary}\label{cor4.5} Let $G$ be a group of type $\mathrm{FP}_\infty$ over $\mathbb{Z}$ such that $\Ghd_\mathbb{Z}G<\infty$. Then, there exists a field $F$ such that $\Ghd_FG=\Ghd_\mathbb{Z}G.$
\end{corollary}

\section{The Gorenstein homological analogue of Fel\texorpdfstring{$'$}{'}dman's theorem}\label{sec:feldman}
The classical theorem of Fel$'$dman \cite{Fel71} (see also \cite{Bieri}) computes the cohomological dimension of a group extension $1\to N\to G\to Q\to1$: while the subadditivity $\cd_kG\le\cd_kN+\cd_kQ$ holds in complete generality, Fel$'$dman's theorem upgrades it to the equality $\cd_kG=\cd_kN+\cd_kQ$ under a condition on the top cohomology $H^n(N,kN)$ of the kernel, $n=\cd_kN$. The aim of this section is a Gorenstein homological analogue. 
Throughout this section, $k$ is a commutative ring, $1\to N\to G\to Q\to1$ is a group extension and $\D=\Hom_\Z(-,\Q/\Z)$. For a $kG$-module $W$ we write $H_q(N,W)=\Tor_q^{kN}(W,k)$; these groups carry a natural $kQ$-module structure (the action of $G$ by conjugation, with the elements of $N$ acting trivially; see \cite[III.8, VII.6]{Brown}). This is the structure appearing in the Lyndon--Hochschild--Serre spectral sequence in homology
\begin{equation}\label{eq:LHS}
E^2_{pq}=H_p\big(Q,H_q(N,W)\big)\Longrightarrow H_{p+q}(G,W).
\end{equation}
We recall two basic properties of the Pontryagin dual (see e.g.\ \cite{EJ}): $\D M$ is injective if and only if $M$ is flat; and $\D R$ is an injective cogenerator of the category of $R$-modules, so that every injective $R$-module is a direct summand of a module of the form $\D(R^{(\alpha)})$ for a suitable cardinal $\alpha$. 

The first aim of this section is to prove the natural isomorphism $H_q(N,\D V)\cong\D H^q(N,V)$ for every $kG$-module $V$ (see Lemma \ref{lem5.3}), where the group $N$ is of type $FP_{\infty}$ over $k$, together with its consequences for coinduced modules.

\begin{lemma}\label{lem5.1}
	Let $W$ and $M$ be two $kG$-modules. Given $g\in G$, $f\in\Hom_{kN}(M,W)$, $\varphi\in\D W$ and $x\in M$, we define $g\cdot f$ and $g\varphi$ by letting
	\[
	(g\cdot f)(x)=g\,f(g^{-1}x)
	\qquad\text{and}\qquad
	(g\varphi)(w)=\varphi(g^{-1}w),\ \ w\in W,
	\]
	and we set $g\cdot(\varphi\otimes x)=g\varphi\otimes gx$ on $\D W\otimes_{kN}M$. Then:
	\begin{enumerate}[label=(\roman*)]
		\item These actions endow $\Hom_{kN}(M,W)$ and $\D W\otimes_{kN}M$ with the structure of a $kG$-module, which is natural in $M$.
		\item The elements of $N$ act trivially, so that we actually obtain an action of the quotient group $Q$ on both modules, natural in $M$.
		\item The evaluation map
		\[
		\theta_M:\ \D W\otimes_{kN}M\longrightarrow\D\Hom_{kN}(M,W),
		\]
		given by $\theta_M(\varphi\otimes x)(f)=\varphi\big(f(x)\big)$ is $kQ$-linear, where $\D\Hom_{kN}(M,W)$ is endowed with the natural contragredient action given by $(g\cdot\psi)(f)=\psi(g^{-1}\cdot f)$.
	\end{enumerate}
\end{lemma}

\begin{proof}
	(i) Since $N$ is normal in $G$, it follows that the map $g\cdot f$ is $kN$-linear. Indeed, for $a\in N$ and $x\in M$ we have
	$(g\cdot f)(ax)=g\,f(g^{-1}ax)=g\,f\big((g^{-1}ag)\,g^{-1}x\big)=g\,(g^{-1}ag)\,f(g^{-1}x)=a\,(g\cdot f)(x)$, since
	the map $f$ is $kN$-linear and $g^{-1}ag$ belongs to $N$. On the tensor product, we regard $\D W$ as a right $kN$-module with action $(\varphi a)(w)=\varphi(aw)$, $w\in W$. Then, for every $a\in N$ we have $g\varphi a=\big(w\mapsto\varphi(a\,g^{-1}w)\big)=\big(w\mapsto\varphi(g^{-1}(gag^{-1})w)\big)=(g\varphi)\,(gag^{-1})$. Consequently, $g\cdot(\varphi a\otimes x)=(g\varphi)(gag^{-1})\otimes gx=g\varphi\otimes(gag^{-1})gx=g\cdot(\varphi\otimes ax)$ and hence the action on $\D W\otimes_{kN}M$ is well defined. The naturality in $M$ follows from the definitions.
	
	(ii) For $a\in N$, the $kN$-linearity of $f$ gives $(a\cdot f)(x)=a\,f(a^{-1}x)=a\,a^{-1}f(x)=f(x)$, $x\in M$. Moreover, $a\varphi=\varphi\,a^{-1}$ and hence $a\cdot(\varphi\otimes x)=\varphi a^{-1}\otimes ax=\varphi\otimes x$. As the action of the elements of $N$ is trivial, we obtain an action of the quotient group $Q$.
	
	(iii) For any $g\in G$, $\varphi\in\D W$ and $x\in M$ we have 
	\begin{align*}
		\theta_M\big(g\varphi\otimes gx\big)(f)&=(g\varphi)\big(f(gx)\big)=\varphi\big(g^{-1}f(gx)\big)=\varphi\big((g^{-1}\cdot f)(x)\big)\\
		&=\theta_M(\varphi\otimes x)\big(g^{-1}\cdot f\big)=\big(g\cdot\theta_M(\varphi\otimes x)\big)(f),
	\end{align*}
	as needed.
\end{proof}
	
\begin{lemma}\label{lem5.2}
	Let $V$ be a $kN$-module. For any $kN$-module $L$ there is an additive map
	\[
	\theta_L:\ \D V\otimes_{kN}L\longrightarrow \D\Hom_{kN}(L,V)
	\]
	which is natural in $L$, given by $\theta_L(\varphi\otimes x)(f)=\varphi\big(f(x)\big)$, for all $\varphi\in \D V$, $x\in L$, $f\in\Hom_{kN}(L,V)$. If, moreover, $L$ is a finitely generated projective $kN$-module, then $\theta_L$ is bijective.
\end{lemma}

\begin{proof}
	The map $\theta_L$ is well defined, since $\theta_L(\varphi a\otimes x)(f)=\varphi\big(a\,f(x)\big)=\varphi\big(f(ax)\big)=\theta_L(\varphi\otimes ax)(f)$ for all $a\in N$, the map $f$ being $kN$-linear. Moreover, the naturality in $L$ follows from the definition. It is clear that $\theta_L$ is bijective in the special case where $L=kN$. Indeed, under the canonical identifications $\D V\otimes_{kN}kN\cong\D V$ ($\varphi\otimes a\mapsto\varphi a$) and $\Hom_{kN}(kN,V)\cong V$ ($f\mapsto f(1)$), the map $\theta_{kN}$ corresponds to the identity map of $\D V$. Moreover, if $L,L'$ are two $kN$-modules, then the additive map $\theta_{L\oplus L'}$ is naturally identified with the direct sum $\theta_L\oplus\theta_{L'}$ of the additive maps $\theta_L$ and $\theta_{L'}$. It follows readily that $\theta_L$ is bijective if $L$ is any finitely generated projective $kN$-module.
\end{proof}

We now consider a projective resolution $\mathbf P$ of the trivial $kG$-module $k$
\[
\cdots\rightarrow P_i\rightarrow P_{i-1}\rightarrow\cdots\rightarrow P_0\rightarrow k\rightarrow 0.
\]
Since any projective $kG$-module is $kN$-projective as well, $\mathbf P$ is also a projective resolution of $k$ as a $kN$-module. Hence, we can compute the homology and cohomology groups of $N$ using $\mathbf P$. The naturality of the $kQ$-module structures of Lemma \ref{lem5.1} shows that $\D W\otimes_{kN}\mathbf P$ and $\Hom_{kN}(\mathbf P,W)$ are complexes of $kQ$-modules. In particular, the groups $H_q(N,\D W)$ and $H^q(N,W)$ are $kQ$-modules as well (independent of the choice of $\mathbf P$), and it is precisely this $kQ$-module structure that appears in \eqref{eq:LHS}. Similarly, the naturality of the $kQ$-linear maps $\theta_{P_i}$ of Lemma \ref{lem5.1}(iii) shows that these are the components of a map between chain complexes of $kQ$-modules
\[
\theta=\theta_{\mathbf P}:\ \D W\otimes_{kN}\mathbf P\rightarrow\D\Hom_{kN}(\mathbf P,W).
\]

\begin{lemma}\label{lem5.3}
	Let $N$ be of type $\mathrm{FP}_\infty$ over $k$, $\mathbf P$ be a projective resolution of the trivial $kG$-module $k$ and $V$ be a $kG$-module. The natural evaluation map
	\[
	\theta:\ \D V\otimes_{kN}\mathbf P\rightarrow \D\Hom_{kN}(\mathbf P,V),
	\qquad \theta(\varphi\otimes x)(f)=\varphi\big(f(x)\big),
	\]
	is a quasi-isomorphism of complexes. In particular, for every $q\ge0$ there is a natural isomorphism of $kQ$-modules
	\[
	H_q(N,\D V)\;\cong\;\D H^q(N,V).
	\]
\end{lemma}

\begin{proof} Since the group $N$ is of type $\mathrm{FP}_\infty$ over $k$, there exists a resolution $\mathbf P'$ of the trivial $kN$-module $k$ which consists of finitely generated projective $kN$-modules in each degree. We know that there exists a homotopy equivalence $u:\mathbf P'\to\mathbf P$ in the category of chain complexes of $kN$-modules. Of course, $u$ induces homotopy equivalences $\D V\otimes u$ and $\D\Hom_{kN}(u,V)$, the additive functors $\D V\otimes_{kN}-$ and $\D\Hom_{kN}(-,V)$ preserving chain homotopies. We now claim that the following diagram of chain complex maps
	\[
	\begin{array}{ccc}
		\D V\otimes_{kN}\mathbf P & \xrightarrow{\ \theta_{\mathbf P}\ } & \D\Hom_{kN}(\mathbf P,V)\\[2pt]
		{\scriptstyle \D V\otimes u}\uparrow\ \ & & \ \ \uparrow{\scriptstyle \D\Hom(u,V)}\\[2pt]
		\D V\otimes_{kN}\mathbf P' & \xrightarrow[\ \cong\ ]{\ \theta_{\mathbf P'}\ } & \D\Hom_{kN}(\mathbf P',V)
	\end{array}
	\]
	is commutative. In other words, we claim that the following diagram of abelian groups 
	
	\[
	\begin{array}{ccc}
		\D V\otimes_{kN}P_i & \xrightarrow{\ \theta_{P_i}\ } & \D\Hom_{kN}(P_i,V)\\[2pt]
		{\scriptstyle \D V\otimes u}\uparrow\ \ & & \ \ \uparrow{\scriptstyle \D\Hom(u,V)}\\[2pt]
		\D V\otimes_{kN}P'_i & \xrightarrow[\ \cong\ ]{\ \theta_{P'_i}\ } & \D\Hom_{kN}(P'_i,V)
	\end{array}
	\]
	is commutative for every $i\geq 0$.
	Indeed, the maps $\theta_{P_i}$ and $\theta_{P'_i}$ are the maps of Lemma \ref{lem5.2} and hence the commutativity follows from the naturality established there. Since the two vertical maps of the diagram are homotopy equivalences and the horizontal map at the bottom is bijective (Lemma \ref{lem5.2}), all of these three chain complex maps are quasi-isomorphisms. It follows readily that the horizontal map at the top is a quasi-isomorphism as well. Finally, we note that the exact functor $\D$ commutes with homology. Hence, for every $q\geq 0$ we deduce that
	\[H_q(N,\D V)=H_q\big(\D V\otimes_{kN}\mathbf P\big)\xrightarrow[\ \cong\ ]{\ \theta\ }H_q\big(\D\Hom_{kN}(\mathbf P,V)\big)\cong\D H^q(N,V),
	\] as needed.
\end{proof}

\begin{corollary}\label{cor5.4}
Let $N$ be of type $\mathrm{FP}_\infty$ over $k$ and $\alpha$ be a cardinal. Then, for every $q\ge0$, there is an isomorphism of $kQ$-modules:
\[
H_q\big(N,\D(kG^{(\alpha)})\big)\;\cong\;\Coind_1^Q \big( \D\big(H^q(N,kN)^{(\alpha)}\big)\big)
=\Hom_k\big(kQ,\D\big(H^q(N,kN)^{(\alpha)}\big)\big).
\]
\end{corollary}

\begin{proof}
Since $N$ is $\mathrm{FP}_\infty$, the functors $H^q(N,-)$ commute with direct sums and hence \cite[Corollary~4.3]{ET22} yields
$H^q\big(N,kG^{(\alpha)}\big)\cong H^q(N,kG)^{(\alpha)}\cong\Ind_1^Q\big(H^q(N,kN)^{(\alpha)}\big).$
Applying Lemma \ref{lem5.3} for $V=kG^{(\alpha)}$ we deduce that $H_q(N,\D(kG^{(\alpha)}))\cong\D H^q(N,kG^{(\alpha)})$. Furthermore, for every $k$-module $X$ we have $\D \big(\Ind_1^Q X\big)=\Hom_\Z(kQ\otimes_kX,\Q/\Z)\cong\Hom_k\big(kQ,\Hom_\Z(X,\Q/\Z)\big)=\Coind_1^Q\big(\D X\big)$. Consequently, we get the following isomorphisms of $kQ$-modules
\begin{align*}
H_q(N,\D(kG^{(\alpha)}))&\cong\D H^q(N,kG^{(\alpha)})\cong \D \big(\Ind_1^Q\big( H^q(N,kN)^{(\alpha)}\big)\big)\\
&\cong \Coind_1^Q\big(\D \big(H^q(N,kN)^{(\alpha)}\big)\big),
\end{align*}
as needed.
\end{proof}

\begin{corollary}\label{cor5.5}
Under the hypotheses of Corollary \ref{cor5.4}:
\begin{enumerate}[label=(\roman*)]
\item If the $k$-module $H^q(N,kN)$ is flat, then $H_q\big(N,\D(kG^{(\alpha)})\big)$ is an injective $kQ$-module.
\item If, for some integer $n\geq 0$, the $k$-module $H^n(N,kN)$ contains $k$ as a direct summand, then the $kQ$-module $H_n\big(N,\D(kG^{(\alpha)})\big)$ contains a copy of the injective module $\D\big(kQ^{(\alpha)}\big)$ as a direct summand.
\end{enumerate}
\end{corollary}

\begin{proof}
(i) If the $k$-module $H^q(N,kN)$ is flat, then so is $H^q(N,kN)^{(\alpha)}$; hence the $k$-module $\D\big(H^q(N,kN)^{(\alpha)}\big)$ is injective and the coinduced module $\Coind_1^Q \big(\D\big(H^q(N,kN)^{(\alpha)}\big)\big)$ is an injective $kQ$-module. Invoking Corollary \ref{cor5.4}, we deduce that $H_q\big(N,\D(kG^{(\alpha)})\big)$ is an injective $kQ$-module, as needed.

(ii) If the $k$-module $H^n(N,kN)$ contains $k$ as a direct summand, then the $k$-module $H^n(N,kN)^{(\alpha)}$ contains $k^{(\alpha)}$ as a direct summand and hence $\D (k^{(\alpha)})$ is a direct summand of $\D\big(H^n(N,kN)^{(\alpha)}\big)$. Thus, $\Coind_1^Q\D\big(k^{(\alpha)}\big)$ is a direct summand of $\Coind_1^Q\big(\D\big(H^n(N,kN)^{(\alpha)}\big)\big)$. From Corollary \ref{cor5.4} and the fact that $\Coind_1^Q\D\big(k^{(\alpha)}\big)\cong \D\big(kQ^{(\alpha)}\big)$, it follows that the $kQ$-module $\D\big(kQ^{(\alpha)}\big)$ is a direct summand of $H_n\big(N,\D(kG^{(\alpha)})\big)$, as needed.
\end{proof}

The next lemma is the homological analogue and a strengthening of the observation preceding \cite[Theorem~4.5]{ET22}. There, the finiteness of $\Gcd_kN$ is required, whereas here the finiteness of $\Ghd_kN$ suffices.

\begin{lemma}\label{lem5.6}
Let $N$ be a group of type $\mathrm{FP}_\infty$ over the commutative ring $k$ such that $\Ghd_kN=n<\infty$. Then, $H^n(N,kN)\neq0$.
\end{lemma}

\begin{proof} Since $\Gfd_{kN}k=n<\infty$, Proposition \ref{prop2.1}(ii) yields the existence of an injective $kN$-module $I$ such that $\Tor_n^{kN}(I,k)\neq0$. Now $I$ is a direct summand of a module $\D\big(kN^{(\beta)}\big)$ for a suitable $\beta$ and hence $\Tor_n^{kN}\big(\D(kN^{(\beta)}),k\big)=H_n\big(N,\D(kN^{(\beta)})\big)\neq0$. By Lemma \ref{lem5.3} we have $H_n\big(N,\D(kN^{(\beta)})\big)\cong \D H^n\big(N,kN^{(\beta)}\big)$. Since $N$ is of type $\mathrm{FP}_\infty$, the cohomology functor $H^n(N,-)$ commutes with direct sums and hence 
$H_n\big(N,\D(kN^{(\beta)})\big)\cong \D\big(H^n(N,kN)^{(\beta)}\big)$. It follows that $\D\big(H^n(N,kN)^{(\beta)}\big)\neq0$. We conclude that $H^n(N,kN)\neq0$, as needed.
\end{proof}

\begin{theorem}\label{theo5.7}
Let $k$ be a commutative ring with $\sfli k<\infty$ and let $1\to N\to G\to Q\to1$ be a group extension such that:
\begin{enumerate}[label=(\roman*)]
\item $N$ is of type $\mathrm{FP}_\infty$ over $k$,
\item $\Ghd_kN=n<\infty$ and $\Ghd_kQ<\infty$,
\item the $k$-modules $H^i(N,kN)$ are flat for $i<n$, and $H^n(N,kN)$ contains a copy of $k$ as a direct $k$-summand.
\end{enumerate}
Then, $\Ghd_kG=\Ghd_kN+\Ghd_kQ$.
\end{theorem}

\begin{proof}Let $\Ghd_kN=n<\infty$ and $\Ghd_kQ=m<\infty$. By \cite[Proposition~4.4]{RY} we have $\Ghd_kG\le n+m <\infty$ and hence it suffices to show that $\Ghd_kG\ge n+m$. Since $\Ghd_kG<\infty$, Proposition \ref{prop2.1}(ii) yields that
\[
\Ghd_kG=\sup\{i:\Tor_i^{kG}(W,k)\neq0,\ W\ \text{an injective }kG\text{-module}\}.
\] Consequently, it suffices to find an injective $kG$-module $W$ such that $H_{n+m}(G,W)=\Tor_{n+m}^{kG}(W,k)\neq0$.
Using again Proposition \ref{prop2.1}(ii) for the dimension $\Gfd_{kQ}k=m<\infty$, we deduce that there exists an injective $kQ$-module $I$ such that $\Tor_m^{kQ}(I,k)\neq0$. We choose now a cardinal $\alpha$ such that $I$ is a direct summand of $\D\big(kQ^{(\alpha)}\big)$, so that
\begin{equation}\label{eq:nonzeroQ}
\Tor_m^{kQ}\big(\D(kQ^{(\alpha)}),k\big)\neq0 .
\end{equation}
We let $W=\D\big(kG^{(\alpha)}\big)$. Since the module $kG^{(\alpha)}$ is flat, $W$ is an injective $kG$-module and its restriction to $kN$ is also an injective $kN$-module. We consider the spectral sequence \eqref{eq:LHS} for $W$ and examine its vanishing pattern. Since $W$ is an injective $kN$-module and $\Gfd_{kN}k=n<\infty$, it follows that $H_q(N,W)=\Tor_q^{kN}(W,k)=0$ for every $q>n$ (see Proposition \ref{prop2.1}(ii)). Consequently $E^2_{pq}=0$ for every $q>n$, in the spectral sequence (\ref{eq:LHS}). We consider now the case where $q<n$ and $p>m$. By hypothesis (iii) and Corollary \ref{cor5.5}(i), the $kQ$-module $H_q(N,W)$ is injective. Since $\Gfd_{kQ}k=m<\infty$, Proposition \ref{prop2.1}(ii) yields $E^2_{pq}=\Tor_p^{kQ}\big(H_q(N,W),k\big)=0$ in the spectral sequence (\ref{eq:LHS}).
The second page of (\ref{eq:LHS}) is therefore concentrated in the rectangle $[0,m]\times[0,n]$ and the horizontal ray $\{(p,n):p\ge0\}$. Consequently,
\[
H_{n+m}(G,W)\;\cong\;E^\infty_{m,n}=E^2_{m,n}=H_m\big(Q,H_n(N,W)\big).
\]
By Corollary \ref{cor5.5}(ii) (hypothesis (iii)), the module $H_n(N,W)$ contains $\D\big(kQ^{(\alpha)}\big)$ as a direct $kQ$-summand, so that $H_m(Q,H_n(N,W))$ contains $\Tor_m^{kQ}\big(\D(kQ^{(\alpha)}),k\big)$ as a direct summand, which is nonzero by \eqref{eq:nonzeroQ}. We conclude that $H_{n+m}(G,W)\neq0$ and hence $\Ghd_kG\ge n+m$, as needed.
\end{proof}

\begin{corollary}\label{cor5.8}
Let $k$ be a commutative ring such that $\sfli k<\infty$ and let $1\to N\to G\to Q\to1$ be a group extension where $N$ is of type $\mathrm{FP}_\infty$ over $k$, $\Ghd_kN$ and $\Ghd_kQ$ are finite, and the $k$-modules $H^i(N,kN)$ are free for every $i\le n$, where $n=\Ghd_kN$. Then, $\Ghd_kG=\Ghd_kN+\Ghd_kQ$.
\end{corollary}

\begin{proof}Let $\Ghd_kN=n<\infty$. By Lemma \ref{lem5.6} we have $H^n(N,kN)\neq0$. Consequently, $H^n(N,kN)$ is a nonzero free $k$-module and hence it contains $k$ as a direct summand. Since hypothesis (iii) of Theorem \ref{theo5.7} holds, the result follows.
\end{proof}

\begin{theorem}[Theorem D]\label{theo5.9}
Let $F$ be a field and let $1\to N\to G\to Q\to1$ be a group extension with $N$ of type $\mathrm{FP}_\infty$ over $F$ and $\Ghd_FQ<\infty$. Then,
\[
\Ghd_FG=\Ghd_FN+\Ghd_FQ.
\]
\end{theorem}

\begin{proof}
If $\Ghd_FN=\infty$, then $\Ghd_FG=\infty$ by Proposition \ref{prop2.2}(ii), and the equality holds trivially. We suppose now that $\Ghd_FN<\infty$. Then, the result is an immediate consequence of Corollary \ref{cor5.8}.
\end{proof}


\begin{remark}\label{rem5.10} Dranishnikov \cite{Dra97} constructed groups $G_1,G_2$ of type $\mathrm{FP}_\infty$ of finite virtual cohomological dimension with $\mathrm{vcd}_\Z(G_1\times G_2)<\mathrm{vcd}_\Z G_1+\mathrm{vcd}_\Z G_2$. Since $G_1,G_2,G_1\times G_2$ are $\mathrm{FP}_\infty$ over $\Z$, Corollary \ref{cor3.5} gives $\Ghd_\Z=\Gcd_\Z=\mathrm{vcd}_\Z$ and hence $\Ghd_\Z(G_1\times G_2)<\Ghd_\Z G_1+\Ghd_\Z G_2.$ Thus, hypothesis (iii) cannot be completely omitted from Theorem \ref{theo5.7}. 
\end{remark}

\section{Application: iterated extensions}\label{sec:appl}

As in \cite{ET22}, for a group $N$ and a positive integer $m$ we define inductively the \emph{iterated $m$-fold extension} $G$ of $N$ by itself: for $m=1$, $G=N$ and for $m>1$, the group $G$ fits into an extension $1\to N\to G\to Q\to1$, where $Q$ is an iterated $(m{-}1)$-fold extension of $N$ by itself.

\begin{theorem}[Theorem E]\label{theo6.1}
Let $N$ be a group of type $\mathrm{FP}_\infty$ over $\Z$, let $m$ be a positive integer and let $G$ be an iterated $m$-fold extension of $N$ by itself. Then,
\[
\Ghd_\Z G=m\,\Ghd_\Z N.
\]
Moreover, $\Ghd_\Z G=\Gcd_\Z G$ and $\Ghd_\Z N=\Gcd_\Z N$.
\end{theorem}

\begin{proof}
Since the class of groups of type $\mathrm{FP}_\infty$ over $\Z$ is closed under extensions (see e.g.\ \cite[Proposition~2.7]{Bieri}), it follows that $G$ is of type $\mathrm{FP}_\infty$ over $\Z$. Moreover, $\sfli\Z\le\wgl\Z=1<\infty$ and hence Corollary \ref{cor3.5} yields $\Ghd_\Z G=\Gcd_\Z G$ and $\Ghd_\Z N=\Gcd_\Z N$. By \cite[Theorem~5.1]{ET22} we have $\Gcd_\Z G=m\,\Gcd_\Z N$ which completes the proof.
\end{proof}

\begin{remark}\label{rem6.2} For an alternative proof of the previous result, we can follow the proof of \cite[Theorem~5.1]{ET22}. If $\Ghd_\Z N=\infty$, Proposition \ref{prop2.2}(ii) gives $\Ghd_\Z G=\infty$. Otherwise, Theorem \ref{theo4.3} yields a field $F$ such that $\Ghd_FN=\Ghd_\Z N$. Using induction on $m$ and Theorem \ref{theo5.9} we deduce that $\Ghd_FG=m\,\Ghd_FN$. Then, the chain of inequalities
\[
m\,\Ghd_\Z N=m\,\Ghd_FN=\Ghd_FG\le\Ghd_\Z G\le m\,\Ghd_\Z N
\]
closes, where the first inequality is Proposition \ref{prop2.2}(iv) and the second follows from a repeated application of Proposition \ref{prop2.2}(iii). We conclude that $\Ghd_\Z G=m\,\Ghd_\Z N$, as needed.
\end{remark}

\begin{corollary}\label{cor6.3}
If $N$ is of type $\mathrm{FP}_\infty$ over $\Z$ and $N^m$ denotes the direct product of $m$ copies of $N$, then $\Ghd_\Z N^m=m\,\Ghd_\Z N$. \qed
\end{corollary}

\begin{corollary}\label{cor6.4}
Let $N$ be of type $\mathrm{FP}_\infty$ over $\Z$ with $\mathrm{vcd}_\Z N=n<\infty$ and let $m$ be a positive integer.
\begin{enumerate}[label=(\roman*)]
\item If $G$ is an iterated $m$-fold extension of $N$ by itself, then $\Ghd_\Z G=mn$.
\item $\Ghd_\Z N^m=\Gcd_\Z N^m=\mathrm{vcd}_\Z N^m=mn$.
\end{enumerate}
\end{corollary}

\begin{proof}
When $\mathrm{vcd}_\Z$ is defined and finite, it equals $\Gcd_\Z$ (see the introduction of \cite{ET22}) and Theorem \ref{theo6.1} applies.
\end{proof}

\begin{remark}\label{rem6.5}
(i) The inequality of Remark \ref{rem5.10} shows that for arbitrary extensions of $\mathrm{FP}_\infty$ groups of finite $\Ghd_\Z$ additivity may fail, the subadditivity being strict.

(ii) All results of this section hold, with the same proofs, when $\Z$ is replaced by any principal ideal domain.
\end{remark}

\end{document}